\documentclass[12pt,a4paper]{amsart}
\pdfoutput=1

\usepackage[T1]{fontenc}
\usepackage{lmodern}
\usepackage[tracking]{microtype}  

\usepackage[utf8]{inputenc}

\usepackage[style=numams, 
            sorting=nyt,
            isbn=false,  
            backref=true, 
            backend=biber]{biblatex}
\usepackage{amsmath,amssymb, amsthm}
\usepackage{mathtools}
\usepackage{enumitem}     

\usepackage[colorlinks, linkcolor=blue, citecolor=blue, pdfa]{hyperref}

\usepackage[hcentering, hscale=0.69, vscale=0.71]{geometry}

\newtheorem{mainthm}{Theorem}
\newtheorem{mainlem}[mainthm]{Lemma}

\swapnumbers

\newtheorem{thm}{Theorem}[section]
\newtheorem{cor}[thm]{Corollary}
\newtheorem{lemma}[thm]{Lemma}

\theoremstyle{definition}

\newlist{enumthm}{enumerate}{1}  
\setlist[enumthm,1]{label=\textup{(\roman*)}}

\renewcommand{\phi}{\varphi}
\renewcommand{\theta}{\vartheta}
\newcommand{\eps}{\varepsilon}

\renewcommand{\geq}{\geqslant}
\renewcommand{\leq}{\leqslant}

\newcommand{\nbd}{\nobreakdash-\hspace{0pt}}  

\newcommand{\defemph}[1]{\textbf{#1}}

\newcommand{\nats}{\mathbb{N}}
\newcommand{\ints}{\mathbb{Z}}
\newcommand{\rats}{\mathbb{Q}}
\newcommand{\reals}{\mathbb{R}}
\newcommand{\compl}{\mathbb{C}}
\newcommand{\quats}{\mathbb{H}}
\newcommand{\crp}[1]{\mathbb{#1}}     

\newcommand{\iso}{\cong}    
\newcommand{\nteq}{\trianglelefteq} 

\DeclarePairedDelimiter{\card}{\lvert}{\rvert} 
\DeclarePairedDelimiter{\fdeg}{\lvert}{\rvert} 
\DeclarePairedDelimiter{\erz}{\langle}{\rangle}

\DeclarePairedDelimiterX{\ipcf}[2]{\lbrack}{\rbrack}{#1,#2}  

\DeclareMathOperator{\Ker}{Ker}        

\DeclareMathOperator{\FS}{\nu_2}       
\DeclareMathOperator{\Gal}{Gal}
\DeclareMathOperator{\Aut}{Aut}
\DeclareMathOperator{\Syl}{Syl}
\DeclareMathOperator{\Z}{\mathbf{Z}}         
\DeclareMathOperator{\C}{\mathbf{C}}     
\DeclareMathOperator{\Irr}{Irr}
\DeclareMathOperator{\Lin}{Lin}
\DeclareMathOperator{\ord}{ord}
\DeclareMathOperator{\mat}{\mathbf{M}}      
\DeclareMathOperator{\enmo}{End} 

\DeclareMathOperator{\R}{\mathbf{R}}   
\DeclareMathOperator{\NKer}{NKer}

\hypersetup{pdfauthor={Frieder Ladisch},
            pdftitle={Groups with a nontrivial nonideal kernel}
            pdfkeywords={Characters, Finite groups, Representations, 
                         Schur indices, Division rings}
           }

\begin{document}

\title{Groups with a nontrivial nonideal kernel}
\author{Frieder Ladisch}
\address{Universität Rostock,
         Institut für Mathematik,         
         18051 Rostock,
         Germany}
\email{frieder.ladisch@uni-rostock.de}
\thanks{Author supported by the DFG (Project: SCHU 1503/6-1)}

\begin{abstract} 
  We classify finite groups $G$,
  such that the group algebra, $\mathbb{Q}G$
  (over the field of rational numbers $\mathbb{Q}$),
  is the direct product of the group algebra $\mathbb{Q}[G/N]$
  of a proper factor group $G/N$, and some division rings.
\end{abstract}

\maketitle

\section{Introduction}

Let $G$ be a finite group 
and $\crp{K}$ a field of characteristic zero.
By Maschke's theorem and Wedderburn-Artin theory,
the group algebra $\crp{K}G$ of $G$ over $\crp{K}$
is a direct product of matrix rings 
over division algebras:
\[ \crp{K}G \iso
    \mat_{d_1}(D_1)
    \times \dotsm \times
    \mat_{d_r}(D_r).
\]
A natural question to ask is when each factor in 
this decomposition is actually a division ring
(equivalently, the group algebra $\crp{K}G$ contains 
no nilpotent elements).
In the classical case where $\crp{K}$ is algebraically closed,
it is well known that
$\crp{K} G$ is a direct product of division rings
if and only if $G$ is abelian.
For $\crp{K}=\rats$, the question was solved
by S.~K.~Seghal~\cite[Theorem~3.5]{Sehgal75}
(see Theorem~\ref{t:qg_divringprod} below).

In this paper, we consider a slightly more general question:
Let $1\neq N \nteq G$ be a normal subgroup.
Then 
\[ \crp{K}G \iso \crp{K}[G/N] \times I,
\]
where the (twosided) ideal $I$ is the kernel of the 
canonical homomorphism $\crp{K}G \to \crp{K}[G/N]$.
Now we ask: 
for which finite groups is there an $N\neq 1$
such that
the ideal $I$ above is a direct product of division rings?
If there is such an $N$, then any nilpotent element 
of $\crp{K}G$ has constant coefficients on cosets of $N$.
Also, only twosided ideals of $\crp{K}G$ can distinguish the elements
of $N$.

The following is just a basic observation,
which allows us to state our results more conveniently.

\begin{mainlem}\label{main:introlemma}
  For each field $\crp{K}$ (of characteristic zero) 
  and each finite group $G$, there is a unique 
  maximal normal subgroup 
  $ N$, denoted by $\NKer_{\crp{K}}(G)$,
  such that the 
  kernel of the map $\crp{K}G \to \crp{K}[G/N]$
  is a direct product of division rings.
\end{mainlem}

We will give a more direct definition of
 $\NKer_{\crp{K}}(G)$ in Section~\ref{sec:defnker}
 below, before we prove Lemma~\ref{main:introlemma}.
We call $\NKer_{\crp{K}}(G)$ the 
\defemph{nonideal kernel} of $G$ (over $\crp{K}$).
 
We view the zero ideal as an empty
product of division rings, so possibly 
$\NKer_{\crp{K}}(G)=1$.  
Indeed, this is the case for ``most'' groups,
and we want to classify the groups $G$ for which
$\NKer_{\crp{K}}(G)\neq 1$.
Our first result concerns the field $\reals$ of real numbers. 

We need to recall a definition:
A nonabelian group $G$ 
is called \defemph{generalized dicyclic},
if it has an abelian subgroup $A$ of index $2$ and an element
$g \in G \setminus A$
such that $g^2 \neq 1$ and $ a^g = a^{-1} $
for all $a\in A$.
If $A$ is cyclic, then $G$ is called \defemph{dicyclic}
(or generalized quaternion).
Furthermore, $Q_8$ denotes the quaternion group
of order $8$ and $C_n$ a cyclic group of order $n$.

\begin{mainthm}\label{main:reals}
  Let $G$ be a finite group.
  Then $ \NKer_{\reals}(G) > 1 $ if and only if
  one of the following holds:
  \begin{enumthm}
  \item $G $ is abelian and $G\neq \{1\}$.
  \item $G$ is generalized dicyclic.
  \item $G \iso C_4 \times Q_8 \times (C_2)^r$, $r\in \nats $.
  \item $G \iso Q_8 \times Q_8 \times (C_2)^r$, $r\in \nats $.
  \end{enumthm}
\end{mainthm}

The motivation for this work is a question of Babai~\cite{Babai77}.
Babai asked which finite groups
are isomorphic to the affine symmetry group of 
an orbit polytope.
(An orbit polytope is a polytope such that its 
(affine) symmetry groups acts transitively on the vertices
of the polytope.)
In joint work with Erik Friese~\cite{FrieseLadisch18}
(continuing our earlier paper~\cite{FrieseLadisch16}), 
we develop a general theory, which shows, among other things, 
that $G$ is isomorphic to the affine symmetry group of an orbit
polytope when
$\NKer_{\reals}(G) = 1$.
When $\NKer_{\reals}(G) > 1$, this may or may not be the case.
Theorem~\ref{main:reals} above is an essential ingredient in our answer
to Babai's question.
Similarly, when $\NKer_{\rats}(G)=1$,
then $G$ can be realized as the affine symmetry group of an 
orbit polytope with vertices having rational coordinates.

The classification of groups
with $\NKer_{\rats}(G) > 1$ is more complicated.
To state it, we first describe a special type of such 
groups.

\begin{mainlem}\label{main:pq}
  Let $p\neq q$ be primes, let
  $P = \erz{g} \times P_0$ be an abelian $p$-group
  and $Q$ an abelian $q$-group.
  Suppose $P$ acts on $Q$ such that
  $x^g = x^k$ for all $x\in Q$
  and some integer $k$ independent of $x\in Q$, 
  and such that $\C_P(Q) = \erz{ g^{p^c} } \times P_0$
  for some $c\geq 1$.
  Suppose that $p^d = \ord( g^{p^c} )$ is the exponent 
  of $\C_P(Q)$,
  and that $(q-1)_p$, the $p$-part of $q-1$,
  divides $p^d$.
  If $p=(q-1)_p =2$,
  assume additionally that $(q^2-1)_2$ divides $2^d$.
  Then for the semidirect product $G=PQ$,
  we have $\NKer_{\rats}(G) \cap \erz{g} \neq 1$.  
\end{mainlem}

Notice that the assumption on the action of $g$ on $Q$
and $\card{P/\C_P(Q)} = p^c$ imply that
$p^c $ divides $q-1$, and that the multiplicative order
of $k$ modulo the exponent of $Q$ is just $p^c$.
As we will see later, 
$1\neq \NKer_{\rats}(G) \leq \erz{g^{p^c}}$.
Whenever we mention ``groups as in 
Lemma~\ref{main:pq}'', we also use the notation 
established in the statement of Lemma~\ref{main:pq}.

\begin{mainthm}\label{main:qclass}
  Let $G$ be a finite group. 
  Then $\NKer_{\rats}(G)\neq 1$
  if and only if at least one of the following holds:
  \begin{enumthm}
  \item \label{it:m_ab}%
        $G$ is abelian and $G \neq \{1\}$.
  \item \label{it:m_nilp}%
        $G = S \times A$,
        where $S$ is a $2$-group of exponent $4$
        which appears on
        the list from Theorem~\ref{main:reals},
        the group $A$ is abelian of odd order,
        and the multiplicative order of\/
        $2$ modulo $\card{A}$ is odd.
  \item \label{it:m_gendic}%
        $G$ is generalized dicyclic.
  \item \label{it:m_allab}%
        $G = (PQ) \times B$,
        where the subgroups
        $P\in \Syl_p(G)$, $Q\in \Syl_q(G)$ and
        $B$ are abelian, 
        $PQ$ is as in Lemma~\ref{main:pq},
        and the $p$-part of the multiplicative order
        of $q$ modulo $\card{B}$ divides the multiplicative
        order of $q$ modulo $p^d$.
  \item \label{it:m_dirprod_h}%
        $G = Q_8 \times (C_2)^r \times H$,
        where $H$ is as in \ref{it:m_allab} and has odd order,
        and the multiplicative order of\/
        $2$ modulo $\card{H}$ is odd.
  \end{enumthm}
\end{mainthm}

Case~\ref{it:m_nilp} contains the groups
$G = Q_8 \times( C_2)^r \times A$,
for which $\rats G$ is a direct product of division rings,
as classified by Sehgal~\cite{Sehgal75}.

An important tool in the proofs
of Theorems~\ref{main:reals} and~\ref{main:qclass}
is Blackburn's classification of finite groups
in which all nonnormal subgroups have a nontrivial 
intersection~\cite{Blackburn66}.
As we will see below, $\NKer_{\crp{K}}(G)$
is always contained in the intersection 
of all nonnormal subgroups of $G$.
While the proof of Theorem~\ref{main:reals} is relatively elementary,
the proof of Theorem~\ref{main:qclass} also depends on some
deep facts about division algebras and Schur indices.
Recently, Erik Friese found a proof of 
Theorem~\ref{main:reals} which does not depend on Blackburn's
classification~\cite[Theorem~6.2.3]{Friese18diss}.

The paper is organized as follows: 
In Section~\ref{sec:skewlin}, we review some basic facts
about representations and characters over fields not necessarily
algebraically closed, and in particular Schur indices.
We also introduce the auxiliary concept of
\emph{skew-linear characters}.
In Section~\ref{sec:defnker}, we define $\NKer_{\crp{K}}(G)$
and prove some elementary properties.
In Section~\ref{sec:dedekind}, we consider Dedekind groups
(groups such that all subgroups are normal).
In such groups, we have either 
$\NKer_{\crp{K}}(G) = 1$, or $\NKer_{\crp{K}}(G)=G$,
where the latter are exactly the groups such that 
$\crp{K}G$ is a direct product of division rings.
Finally, Section~\ref{sec:classr} 
contains the proof of Theorem~\ref{main:reals},
and Section~\ref{sec:classq} the (long) proof of
Theorem~\ref{main:qclass}.

\section{Skew-linear characters}
\label{sec:skewlin}

Let $G$ be a finite group. 
For simplicity, assume that $\crp{K}\subseteq \compl$
and write $\Irr G$ for the set of irreducible complex
characters of $G$.
We begin by reviewing the relation between the representation
theory of $G$ over $\crp{K}$ and over 
$\compl$~\cite[\S~38]{huppCT}\cite[Chapter~10]{isaCTdov}.

By Maschke's theorem and general Wedderburn-Artin theory,
the group algebra $\crp{K}G$ is the direct product
of simple rings:
\[ \crp{K}G = A_1 \times \dotsm \times A_r.
\]
Each $A_i$ is a simple ideal,
and the set of the $A_i$'s is uniquely determined as the set
of simple ideals of $\crp{K}G$.
The $A_i$'s are called the 
\defemph{block ideals of $\crp{K}G$}.
Each $A_i$ is generated by a centrally primitive idempotent
$e\in \Z(\crp{K}G)$.
By Wedderburn-Artin theory,
each $A_i$ is isomorphic to a matrix ring
over a division ring.

We now relate the above decomposition
to the complex irreducible characters of $G$.
Recall that the \defemph{Schur index}
of $\chi\in \Irr G$ over $\crp{K}$
is the smallest positive integer $m = m_{\crp{K}}(\chi)$ 
such that
$m\chi $ is afforded by a representation with entries in
$\crp{K}(\chi)$, the field generated by $\crp{K}$
and the values of $\chi$.

\begin{lemma}\label{l:reps_basics}
  Let $\chi\in \Irr G$.
  \begin{enumthm}
  \item \label{it:uniblock}
        There is a unique block ideal $A = A_{\crp{K}}(\chi)$ 
        of\/ $\crp{K}G$
        such that $\chi(A)\neq 0$.
  \item \label{it:gal_conj} 
        Let $\psi\in \Irr G$.
        Then
        $\psi(A)\neq 0$ if and only if
        $\psi$ and $\chi$ are Galois conjugate over $\crp{K}$,
        that is, $\psi=\chi^{\alpha}$ for some
        $\alpha\in \Gal( \crp{K}(\chi)/\crp{K} )$.
  \item \label{it:cent_bl}
        Write $A\iso \mat_n(D)$ for some 
        division ring~$D=D_{\crp{K}}(\chi)$.
        Then $\Z(A)\iso \Z(D)\iso \crp{K}(\chi)$.
  \item \label{it:si_deg}
        $\fdeg{D:\Z(D)} = m_{\crp{K}}(\chi)^2$
        and $\chi(1) = nm_{\crp{K}}(\chi)$.
  \end{enumthm}
\end{lemma}
\begin{proof}
  This is standard~\cite[Theorems~38.1 and~38.15]{huppCT}.
\end{proof}

It follows from Lemma~\ref{l:reps_basics}
that $A = A_{\crp{K}}(\chi)$ is itself a division ring if and only if 
$\chi(1)=m_{\crp{K}}(\chi)$.
In this case, the projection
$\crp{K}G\to A$ defines 
a homomorphism $\phi$
from $G$ into the multiplicative group of $D$.
Notice also that
$\Ker(\phi)=\Ker(\chi)$.
For this reason,
we call a character $\chi$ 
\defemph{skew-linear} (over $\crp{K}$),
if $\chi(1) = m_{\crp{K}}(\chi)$.
Thus skew-linear characters generalize linear characters.
Since $m_{\compl}(\chi)=1$ for all $\chi$, 
skew-linear over $\compl$ is the same as linear.

If $\chi\in \Irr(G)$ is linear, then (trivially)
the reduction to any subgroup
is irreducible and linear.
This fact generalizes to skew-linear characters as follows:

\begin{lemma}\label{l:res_skewlin}
  Let $\chi\in \Irr(G)$ be skew-linear over 
  the field\/~$\crp{K}$, and $H\leq G$.
  Then the irreducible constituents of $\chi_H$
  are skew-linear over $\crp{K}$,
  and are Galois conjugate over the field
  $\crp{K}(\chi)$.
\end{lemma}
\begin{proof}
   Let $\theta\in \Irr(H)$ be a constituent of $\chi_H$.
      Then~\cite[Lemma~10.4]{isaCTdov}
      \[ m_{\crp{K}}(\chi)
        \quad \text{divides} \quad 
        \ipcf{\chi_H}{\theta} 
        \cdot
        \card{\crp{K}(\chi, \theta):\crp{K}(\chi)}
        \cdot 
        m_{\crp{K}}(\theta).
      \]
      
      Let $\sigma \in \Gal( \crp{K}(\chi,\theta)/\crp{K}(\chi))$.
      Then $\ipcf{\chi_H}{\theta^{\sigma}} = \ipcf{\chi_H}{\theta}$.
      Thus each of the 
      $\card{\crp{K}(\chi,\theta):\crp{K}(\chi)}$
      characters $\theta^{\sigma}$ occurs in $\chi_H$ with multiplicity
      $\ipcf{\chi_H}{\theta}$.
      It follows that
      \begin{align*} 
        \ipcf{\chi_H}{\theta}
        \cdot 
        \card{\crp{K}(\chi, \theta):\crp{K}(\chi)}
        \cdot
        \theta(1)
        &\leq \chi(1) 
        = m_{\crp{K}}(\chi)  \\
        &\leq \ipcf{\chi_H}{\theta} 
              \cdot
              \card{\crp{K}(\chi, \theta):\crp{K}(\chi)}
              \cdot
              m_{\crp{K}}(\theta)
        \\
        & \leq 
        \ipcf{\chi_H}{\theta}
                \cdot 
                \card{\crp{K}(\chi, \theta):\crp{K}(\chi)}
                \cdot
                \theta(1).                
      \end{align*}
      This implies that equality holds throughout,
      in particular, $\theta(1) = m_{\crp{K}}(\theta)$
      and $\chi_H = \ipcf{\chi_H}{\theta} \sum \theta^{\sigma}$,
      the sum running over 
      $\sigma \in \Gal( \crp{K}(\chi,\theta)/\crp{K}(\chi))$.  
\end{proof}

In the rest of this section, we record some
(mostly well known) facts about Schur indices
and blocks of group algebras for later reference.

Recall that
\[ e_{\chi}
   = \frac{ \chi(1) }{ \card{G} }
     \sum_{g\in G} \chi( g^{-1} ) g
\]
is the centrally primitive idempotent
in $\compl G$ corresponding to $\chi\in \Irr G$.
The following simple observation will sometimes be useful.
Notice that it provides an alternative proof of 
$\Z(A)\iso \crp{K}(\chi)$.

\begin{lemma}\label{l:charfield_bl}
  Let $\chi\in \Irr G$
  and let $A$ be the block ideal of\/ $\crp{K}G$
  such that $\chi(A)\neq 0$.
  Then
  \[ A \iso \crp{K}(\chi)Ge_{\chi}
      \quad\text{by}\quad
     A \ni a \mapsto ae_{\chi}. 
  \]
\end{lemma}
\begin{proof}
  Set 
  \[ e:=\sum_{\alpha \in \Gal(\crp{K}(\chi)/\crp{K})}
     e_{\chi^{\alpha}}.
  \]  
  We claim that $A = \crp{K}Ge$.
  We can decompose $1$ into a sum of primitive idempotents
  in $\Z(\crp{K}G) $,
  and then decompose further in $\Z(\compl G)$.
  Thus there is a unique primitive idempotent $f$
  in $\Z(\crp{K}G)$
  such that $fe_{\chi} = e_{\chi}$.
  But then also
  $fe_{\chi^{\alpha}} = e_{\chi^{\alpha}}$
  for all $\alpha \in \Gal(\crp{K}(\chi)/\crp{K})$
  and thus $fe=e$.
  On the other hand, $e_{\chi}\in \crp{K}(\chi) G$ and
  $e\in \crp{K}G$, and thus $f=e$.
  This shows that $A=\crp{K}Ge$ as claimed.
  
  For $\alpha\in \Gal(\crp{K}(\chi)/\crp{K})$,
  \[ b = \sum_g b_g g
     \in
     \crp{K}(\chi)Ge_{\chi} 
     \quad \text{implies}\quad
     b^{\alpha}:= \sum_{g} b_g^{\alpha} g
     \in \crp{K}(\chi)G e_{\chi^{\alpha}}.
  \]
  Using this, it is straightforward to check that
  \[ \crp{K}(\chi) G e_{\chi}
     \ni b 
     \mapsto 
     \sum_{\alpha\in \Gal( \crp{K}(\chi)/\crp{K} ) }
     b^{\alpha} \in \crp{K} G e
  \]
  yields the inverse of the map
  $a\mapsto ae_{\chi}$.
\end{proof}

Since we will often have to consider characters 
of direct products of groups,
and the corresponding blocks of the group 
algebra, we record the following for later reference.

\begin{lemma}\label{l:blocksdirprod}
  Let $G = U \times V$ be a direct product of groups,
  $\sigma\in \Irr U$ and $\tau\in \Irr V$.
  Then 
  $\chi = \sigma \times \tau\in \Irr G$
  and 
  $\crp{K}(\chi) = \crp{K}(\sigma,\tau)$.
  Let 
  $ A_{ \crp{K} } ( \chi ) $ 
  be the block ideal of\/ $\crp{K}G$ corresponding to $\chi$,
  and 
  $ A_{ \crp{K} } ( \sigma ) $ and 
  $ A_{ \crp{K} } ( \tau ) $
  the block ideals of\/ $\crp{K}U$ and\/ $\crp{K}V$
  corresponding to $\sigma$ and $\tau$.
  Then
  \[ A_{\crp{K}}(\chi)
      \iso \big( A_{\crp{K}}(\sigma) 
               \otimes_{ \crp{K}(\sigma) } 
               \crp{K}(\chi)
            \big)
            \otimes_{ \crp{K}(\chi) }
            \big( A_{\crp{K}}(\tau) 
               \otimes_{ \crp{K}(\tau) }
               \crp{K}(\chi)
            \big).
  \]
\end{lemma}
\begin{proof}
  The irreducible characters of $U\times V$
  are exactly the characters of the form
  $\chi= \sigma \times \tau$, 
  with $\sigma\in \Irr U$ and 
  $\tau \in \Irr V$ \cite[Theorem~4.21]{isaCTdov}.
  Since $\chi((u,1_V))= \sigma(u)\tau(1_V)$ for $u\in U$
  and similarly $\chi((1_U,v))=\sigma(1_U)\tau(v)$
  for $v\in V$,
  we see that $\crp{K}(\chi) = \crp{K}(\sigma,\tau)$.
  
  Set $\crp{L}=\crp{K}(\chi)$.
  The natural isomorphism
  \[ \crp{L}U \otimes_{\crp{L}} \crp{L}V
     \to \crp{L}G,
     \quad 
     \sum_u a_u u \otimes \sum_v b_u v 
     \mapsto \sum_{u,v} a_ub_v(u,v)
  \]
  sends $e_{\sigma}\otimes e_{\tau}$ to 
  $e_{\chi}$
  and thus induces an isomorphism
  \[ \crp{L}Ue_{\sigma}
     \otimes_{\crp{L}}
     \crp{L}Ve_{\tau}
     \iso \crp{L}G e_{\chi}
  \]
  (by comparing dimensions).
  By Lemma~\ref{l:charfield_bl},
  the right hand side is isomorphic to
  $A_{\crp{K}}(\chi)$,
  and on the left hand side
  we have
  \[ \crp{L}U e_{\sigma}
     \iso 
     \crp{K}(\sigma)U e_{\sigma} 
       \otimes_{ \crp{K}(\sigma) }
       \crp{L}
     \iso 
     A_{\crp{K}}(\sigma)
       \otimes_{ \crp{K}(\sigma) }
       \crp{L},      
  \]
  and similarly for the other factor.
  The result follows.
\end{proof}

In Section~\ref{sec:classq},
we need several deep facts about Schur indices,
which we collect now.
For a prime $q$, we write
$m_q(\chi) := m_{\rats_q}(\chi)$,
where $\rats_q$ denotes the field of $q$-adic numbers.
Sometimes, it will be convenient to use this notation
also for the ``infinite prime'', that is,
$m_{\infty}(\chi):= m_{\reals}(\chi)$.

\begin{lemma}\label{l:sifacts}
  Let $\chi\in \Irr(G)$.
  \begin{enumthm}
  \item \label{it:globloc}
       $m_{\rats}(\chi)$ is the least common multiple
       of the local indices
       $m_q(\chi)$, where $q$ runs through all
       primes, including the infinite one.
       \cite[(32.19)]{reinerMO} 
  \item \label{it:localsi_div}
       $m_{\reals}(\chi)$ and $m_2(\chi)$ divide $2$,
       and $m_q(\chi)$ divides $q-1$ for odd $q$.
       \cite[Theorem~4.3, Corollary~5.5]{Yamada74}
  \item \label{it:localsi_brauer}
       Let $\phi$ be an irreducible Brauer character 
       for the prime $q$, and $d_{\chi \phi}$
       the decomposition number.
       Then 
       $m_q(\chi)$ divides 
       $d_{\chi \phi} \card{\rats_q(\chi,\phi):\rats_q(\chi)}$.
       \cite[Theorem~IV.9.3]{Feit82}
  \item \label{it:localsitriv}
       If the finite prime $q$ does not divide $\card{G}$,
       then $m_q(\chi) = 1$.
       \cite[Corollary~IV.9.5]{Feit82}
  \end{enumthm}
\end{lemma}

\begin{cor}\label{c:localindequ}
  Let $\chi\in \Irr(G)$ with 
  $\chi(1) = m_q(\chi)$, 
  where $q$ is a prime number.
  If $H\leq G$ has order not divisible by $q$,
  then any constituent of $\chi_H$ is linear. 
\end{cor}
\begin{proof}
  This is immediate from Lemma~\ref{l:res_skewlin}
  and Lemma~\ref{l:sifacts}~\ref{it:localsitriv}.
\end{proof}

\section{The nonideal kernel}
\label{sec:defnker}

For every field $\crp{K}$ 
and any finite group $G$, we define 
\[ \NKer_{\crp{K}}(G)
   := \bigcap \{ \Ker(\chi) 
                 \mid 
                 \chi \in \Irr(G)
                 ,\;
                 \chi(1) > m_{\crp{K}}(\chi)
              \}.
\]
If $m_{\crp{K}}(\chi) = \chi(1)$ for every $\chi\in \Irr(G)$,
we set $\NKer_{\crp{K}}(G):=G$.
We call $\NKer_{\crp{K}}(G)$ the 
\defemph{nonideal kernel} of $G$ over $\crp{K}$.
Notice that $\NKer_{\crp{K}}(G)$, for any field $\crp{K}$, 
is characteristic in $G$.

\begin{lemma}\label{l:subfields}
  Let $\crp{K} \subseteq \crp{L}$ be fields.
  Then $\NKer_{\crp{L}}(G) \subseteq \NKer_{\crp{K}}(G)$.
\end{lemma}
\begin{proof}
  Since $m_{\crp{L}}(\chi)$ divides 
  $m_{\crp{K}}(\chi)$ for any $\chi\in \Irr G$,
  any character which is skew-linear over $\crp{L}$,
  is also skew-linear over $\crp{K}$.
  The result follows.
\end{proof}

\begin{lemma}\label{l:kernelnonlin}
  Let $G$ be a nonabelian group.
  Then
  \[ \bigcap_{\substack{\chi\in \Irr G \\ \chi(1)>1}}
      \Ker \chi = \{1\}.
  \]
\end{lemma}
\begin{proof}
  Suppose that $g\neq 1$ is 
  contained in the kernel of all nonlinear characters.
  Then, by the second orthogonality relation~\cite[(2.18)]{isaCTdov},
  \begin{align*}
    0 &= \sum_{\chi\in \Irr G} \chi(1)\chi(g)
       = \sum_{\substack{\chi \in \Irr G \\ \chi(1)>1}} 
         \chi(1)^2
         + \sum_{\chi\in \Lin G} \chi(1)\chi(g).
  \end{align*}
  The second sum runs over the irreducible characters of $G/G'$
  and has value $\card{G:G'}$ or $0$, 
  according to whether $g\in G'$ or not.
  It follows that
  the first sum must be empty.
  Thus $G$ has no nonlinear characters,
  which means that $G$ is abelian, as claimed.
\end{proof}

\begin{cor}
  Let $G$ be a nonabelian group. 
  Then $\NKer_{\compl}(G) = 1$.
\end{cor}

Let us say that a character $\alpha$
(not necessarily irreducible)
is \defemph{strictly nonideal},
if $\ipcf{\alpha}{\chi} < \chi(1)$ for all
$\chi\in \Irr G$.
(Such a character is afforded by a left ideal of the group algebra,
 which does not contain any nonzero two-sided ideal.)
If at the same time, $\alpha$ is the character of a 
representation with entries in $\crp{K}$, then 
$m_{\crp{K}}(\chi)$ divides $\ipcf{\alpha}{\chi}$
for all $\chi\in \Irr G$~\cite[Corollary~10.2(c)]{isaCTdov}.
Thus no constituent
of $\alpha$ can be skew-linear over $\crp{K}$.
Conversely, if $S$ is a set of non-skew-linear characters over $\crp{K}$,
then we may add the characters of the corresponding 
irreducible representations
over $\crp{K}$
and get a strictly nonideal character $\alpha$
which is afforded by a $\crp{K}$\nbd representation.
Since $\Ker \alpha  = \bigcap \Ker\chi$, where $\chi$ runs through
the constituents of $\alpha$,
it follows that every group $G$ has a strictly nonideal character
$\alpha$ with $\Ker \alpha = \NKer_{\crp{K}}(G)$,
and such that $\alpha$ is afforded by a representation
over $\crp{K}$.
(In the case where $G=\NKer_{\crp{K}}(G)$,
the only such character is $\alpha=0$, however.)

\begin{lemma}\label{l:subgroupskernel}
  Let $H\leq G$ with
  $N:=\NKer_{\crp{K}}( H ) < H$.
  Then
  \[ \NKer_{\crp{K}}(G) \leq \bigcap_{g\in G} N^g 
                        \leq \NKer_{\crp{K}}(H).
  \]
\end{lemma}
\begin{proof}
  Let $\alpha$ be a strictly nonideal character of $H$
  with $N = \Ker \alpha$ and which is afforded by a representation
  over $\crp{K}$.
  Then $0\neq \alpha^G$ is afforded by a representation over $\crp{K}$
  and has kernel
  $\bigcap_{g\in G} N^g$~\cite[Lemma~5.11]{isaCTdov}.
  
  Let $\rho_G$ be the regular character of $G$.
  Notice that a character $\beta$ is strictly nonideal if and only if
  $\rho_G-\beta$ is a character and
  $\ipcf{\rho_G-\beta}{\chi}> 0$ for \emph{all}
  $\chi\in \Irr G$.
  Since $\rho_G = (\rho_H)^G$, we have
  that $\rho_G -\alpha^G = (\rho_H-\alpha)^G$ is a character, and
  \[ \ipcf{\rho_G-\alpha^G}{\chi}
      = \ipcf{(\rho_H-\alpha)^G}{\chi}
      = \ipcf{\rho_H-\alpha}{\chi_H}_H
      > 0
  \]
  for all $\chi\in \Irr G$.
  Thus $\alpha^G$ is strictly nonideal.
\end{proof}

\begin{lemma}\label{l:nt_id}
  Let $N$ be a normal subgroup of $G$, and set
  \[ e_N = \frac{1}{\card{N}} \sum_{n\in N} n.
  \]
  Then $\crp{K}G e_N \iso \crp{K}[G/N]$.
  If $\chi\in \Irr G$, then
  $\chi(e_N)\neq 0 $ if and only if
  $N\leq \Ker(\chi)$.
\end{lemma}
\begin{proof}
  This is well known:
  The canonical epimorphism
  $\crp{K}G \to \crp{K}[G/N]$ is split
  by the map sending a coset $Ng$ to 
  $(1/\card{N})\sum Ng = e_N g$.
  This proves the first statement.
  
  If $N\leq \Ker(\chi)$, then any representation
  affording $\chi$ sends 
  $e_N$ to the identity map.
  If $N\not\leq \Ker(\chi)$,
  then any representation affording $\chi$
  must send $e_N$ to $0$.   
\end{proof}

\begin{lemma}\label{l:idemp}
  Let $N:= \NKer_{\crp{K}}(G)$, and 
  $e_N$ as in Lemma~\ref{l:nt_id}. 
  Then $\crp{K}G(1-e_N)$ is a direct product of division rings.
  In particular, every idempotent $f\in \crp{K}G$ with
  $fe_N = 0$ is central.
\end{lemma}
\begin{proof}
  By Lemma~\ref{l:nt_id}, it follows that
  $ \crp{K} G (1-e_N) $ is the direct product of the block ideals
  which correspond to $\chi\in \Irr G$ 
  with $N\not\leq \Ker(\chi)$.
  By definition of $N$,
  any such $\chi$ is skew-linear over $\crp{K}$,
  and thus the corresponding block ideal
  is a division ring.
  
  In a direct product of division rings, every idempotent 
  is central.
\end{proof}

\begin{proof}[Proof of Lemma~\ref{main:introlemma}]
  The first part of Lemma~\ref{main:introlemma} 
  is contained in Lemma~\ref{l:idemp}.
  Conversely, if $\crp{K}G(1-e_N)$ is a direct product
  of division rings,
  then the above considerations yield 
  that when $m_{\crp{K}}(\chi) < \chi(1)$, we must have
  $N \subseteq \Ker(\chi)$, and thus 
  $N \leq \NKer_{\crp{K}}(G)$.
\end{proof}

Following Blackburn~\cite{Blackburn66},
for any group $G$, we set
\[ \R(G)
    := \bigcap \{U\leq G \mid 
                 U \text{ not normal in } G
               \}.
\]
If every subgroup of $G$ is normal, then we set $\R(G) = G$.
Blackburn~\cite{Blackburn66} classified finite groups in which 
$\R(G) \neq 1$.
Therefore, a group $G$ with $\R(G) \neq 1$ is called a
\defemph{Blackburn group}.
The following result shows why this is relevant for us:

\begin{lemma}\label{l:many_normals}
  For any  finite group $G$ and field $\crp{K}$ 
  of characteristic zero,
  we have $\NKer_{\crp{K}}(G) \leq \R(G)$.
\end{lemma}
\begin{proof}
  Suppose $U\leq G$ is such that
  $N:= \NKer_{\crp{K}}(G)\not\subseteq U \leq G$.
  We need to show that $U \nteq G$.
  
  Set $f= (1-e_N)e_U$, with $e_N$ as above,
  and $e_U:= (1/\card{U})\sum_{u\in U} u$
  analogously.
  Then $f^2 = f \in \crp{K}G(1-e_N)$, 
  since $e_N$ is central in $\crp{K} G$.
  Thus $f$ 
  is central in $\crp{K} G$ by Lemma~\ref{l:idemp}.
  We compute
  \[ f = \frac{1}{\card{U}} \sum_{u\in U} u 
         - \frac{1}{\card{NU}} \sum_{x\in NU} x.
  \]
  As $N$ is not contained in $U$, we have
  $U < NU$.
  As $g^{-1}fg = f$ for all $g\in G$,
  it follows that $U\nteq G$.  
\end{proof}

\section{Dedekind groups}
\label{sec:dedekind}

In this section, we compute 
$\NKer_{\crp{K}}(G)$ for Dedekind groups, 
and determine when $\crp{K}G$ is a direct product 
of division rings.
These results are mostly known.

Recall that a \defemph{Dedekind group}
is a finite group in which all subgroups are 
normal.
First, we recall Dedekind's classification of these 
groups~\cite[Satz~III.7.12 on p.~308]{huppEG1}.

\begin{thm}[Dedekind 1897]\label{t:dedekind}
  Let $G$ be a finite group, such that every 
  subgroup of $G$ is normal.
  Then either $G$ is abelian, or
  \[ G \iso Q_8 \times (C_2)^r \times A \quad(r\geq 0),
  \]
  where $A$ is abelian of odd order.
\end{thm}

Let $\tau\in \Irr(Q_8)$ be the irreducible, faithful 
character of degree $2$.
Then $\quats := \rats Q_8 e_{\tau}$ is a division ring,
the rational quaternions.
$\quats$ can also be described as the 
$\rats$-vector space with basis
$\{1, i, j, k\}$ and multiplication 
defined by $i^2 = j^2 =-1$, $k= ij = -ji$.

\begin{thm}\label{t:kg_divringprod}
    Let $\crp{K}$ be a field and $G$ be a group. 
    Then $\crp{K}G$ is a direct product of division rings
    if and only if either 
    $G$ is abelian, or
    $G\iso Q_8 \times (C_2)^r \times A$,
    where $A$ is abelian of odd order,
    and\/ $\quats \otimes_{\rats}\crp{K}(\lambda)$
    is a division ring for all $\lambda\in \Lin(A)$.
\end{thm}
\begin{proof}
    Suppose $\crp{K} G$ is a direct product of division rings.
    Then all subgroups of $G$ are normal in $G$
    by Lemma~\ref{l:many_normals} 
    (as $\NKer_{\crp{K}}(G) = G$, or directly from the argument in
     the proof of Lemma~\ref{l:many_normals}).
    It follows that either $G$ is abelian, or 
    $G \iso Q_8 \times (C_2)^r \times A$
    with $A$ abelian of odd order.
    
    In the second case,
    let $\tau\in \Irr(Q_8)$ be the irreducible, faithful 
    character of degree $2$.
    Then 
    \[ \crp{K} Q_8 e_{\tau} 
       \iso \quats \otimes_{\rats} \crp{K},
    \]
    the quaternions over $\crp{K}$.
    Any nonlinear, irreducible character of 
    $G = Q_8 \times (C_2)^r \times A$
    has the form $\chi = \tau \times \sigma \times \lambda$,
    where $\sigma \in \Lin(C_2)^r$ and $\lambda\in \Lin A$.
    The corresponding block ideal of the rational group algebra
    is, by Lemma~\ref{l:blocksdirprod}, isomorphic to
    \[ 
      \quats \otimes_{\rats} \crp{K}(\lambda).
    \]
    The result follows.  
\end{proof}

\begin{cor}\label{c:dede_nker}
  Let $G$ be a Dedekind group and $\crp{K}$ a field.
  Then\/ $\NKer_{\crp{K}}(G)=G$ 
  or\/ $\NKer_{\crp{K}}(G) = 1$.
\end{cor}
\begin{proof}
  Suppose that $\NKer_{\crp{K}}(G)\neq G$,
  which means that $\crp{K}G$ is not 
  a direct product of division rings.
  Then $G$ is not abelian and thus
  $G \iso Q_8 \times (C_2)^r \times A$ with 
    $A$ abelian of odd order.  
  Moreover, there is some $\lambda \in \Lin(A)$
  such that $\quats \otimes_{\rats}\crp{K}(\lambda)$ 
  is not a division ring.
  As before, let $\tau\in \Irr(Q_8)$ be the faithful irreducible
  character of $Q_8$.
  Then
  $\Ker(\tau\times 1 \times  \lambda) 
  = 1 \times (C_2)^r \times \Ker(\lambda)$.
  
  It follows that
  $\NKer_{\crp{K}}(G) \subseteq 1 \times \Ker(\mu)$ for every
  $\mu\in \Lin( (C_2)^r \times A)$ such that 
  $\ord(\lambda)$ divides the order of $\mu$.
  Since $A$ contains elements of order $\ord(\lambda)$, we
  see that $\NKer_{\crp{K}}(G) = 1$.  
\end{proof}

Notice that for a linear character
$\lambda$, we have 
$\crp{K}(\lambda) = \crp{K}(\eps_n)$,
where $\eps_n$ is a primitive $n$-th root of unity
and $n= \ord(\lambda)$.
The following lemma collects some results.
These will be needed also in the proof of
Theorem~\ref{main:qclass}.

\begin{lemma}
\label{l:split_quat}
  \hfill
  \begin{enumthm}
  \item \label{it:squaresum}
       $\quats \otimes_{\rats} \crp{K}$
     is a division ring if and only if
     $-1$ is not a sum of two squares in $\crp{K}$.
  \item \label{it:qu_local} 
       $\quats \otimes_{\rats} \rats_2$
       and\/ $\quats \otimes_{\rats} \reals$
       are division rings,
       and\/ 
       $\quats \otimes_{\rats} \rats_p$ for $p$ odd
       is not a division ring.
       (Here $\rats_p$ is the field of $p$\nbd adic numbers.)
  \end{enumthm}
  Let $\eps_n$ be a primitive $n$\nbd th root of unity,
  where $n$ is odd.
  Then
  \begin{enumthm}[resume]
  \item \label{it:qu_cyc}
        $\quats \otimes_{\rats}{\rats(\eps_n)} $
        is a division ring if and only if
        the multiplicative order of $2$ in $(\ints/n)^*$ 
        is odd, if and only if
        $\quats \otimes_{\rats} \rats_2(\eps_n)$
        is a division ring.
  \item \label{it:qu_qu}
        $\quats \otimes_{\rats} \rats( \sqrt{2}, \eps_n)$
        is a division ring only when $n=1$.                
  \end{enumthm}
\end{lemma}
\begin{proof}
  \ref{it:squaresum} and~\ref{it:qu_local}
  are well 
  known \cite[Example~38.13(a)]{huppCT}
  \cite[Ch.~III, Théorème~1]{SerreCA}.
  Assertion~\ref{it:qu_cyc} is a result of 
  Moser~\cite{Moser73}.
  (This can be shown without using the 
  Hasse-Minkowski principle:
  If the residue class of $2$ in $(\ints/p)^*$ has even multiplicative 
    order $2r$,
  then $2^r \equiv -1 \mod p$, and thus 
    $p$ divides $2^r+1$.
  Then an elementary argument shows that
  $-1$ is a sum of two squares in
  $\rats(\eps_p)$~\cite[Example~38.13(d)]{huppCT}.
  If $2$ has odd order in $(\ints/n\ints)$,
  then $\quats \otimes_{\rats}\rats_2(\eps_n)$
  is also a division ring.)
  
  To see~\ref{it:qu_qu}, assume that $n>1$.
  We have to show that 
  $-1$ is a sum of two squares in
  $\crp{K}:=\rats(\sqrt{2}, \eps_n)$.
  By the Hasse-Minkowski principle,
  it suffices to show that $-1$ is a square
  in each possible completion of $\crp{K}$.
  Since $n>1$, $\crp{K}$ can not be embedded into $\reals$.
  If $p$ is odd, then $-1$ is a sum of two squares
  in $\rats_p$ already.
  Finally, $\rats_2(\sqrt{2})$ is a quadratic extension 
  of $\rats_2$
  and thus a splitting field
  of $\quats$~\cite[Lemma~VI.2.14]{Lam05_IQFF}.
  (We notice that in~\ref{it:qu_qu},
  we can replace $\rats(\sqrt{2})$ by any field
  such that the completions at all prime ideals over $2$
  yield extensions of \emph{even} degree over $\rats_2$.)
\end{proof}

As a consequence, we get the following results.

\begin{thm}[Sehgal 1975~\cite{Sehgal75}]
  \label{t:qg_divringprod}
  The group algebra $\rats G$ is a direct product of
  division rings if and only if one of the following holds:
  \begin{enumthm}
  \item $G$ is abelian.
  \item \label{it:m_q8ab}
        $G \iso Q_8 \times (C_2)^r \times A$,
        where $r\geq 0$,
        and $A$ is abelian of odd order, and
        the multiplicative order of $2$ in $(\ints/\card{A})^*$ 
        is odd.
  \end{enumthm}
\end{thm}

\begin{thm}\label{t:local_dpdivrings}\hfill
  \begin{enumthm}
  \item $\rats_2 G$ is a direct product of division rings 
        if and only if\/ $\rats G$ is a direct product of 
        division rings.
  \item Let $p$ be an odd prime.
        Then $\rats_p G$ is a direct product of division rings
        if and only if $G$ is abelian.
  \item $\reals G$ is a direct product of division rings
        if and only if either $G$ is abelian,
        or $G \iso Q_8 \times (C_2)^r$ for some
        $r\geq 0$.
  \end{enumthm}
\end{thm}

\section{Classification over the reals}
\label{sec:classr}

In this section, we prove Theorem~\ref{main:reals}.
We start with a simple observation.

\begin{lemma}\label{l:nn_ord4}
  Suppose that $\NKer_{\reals}(G) \neq 1$,
  and $\erz{g}\not\nteq G$.
  Then $g$ has order $4$,
  and $\NKer_{\reals}(G) = \R(G) = \erz{g^2}$
  has order $2$.
\end{lemma}
\begin{proof}
  By Lemma~\ref{l:many_normals} and the definition of $\R(G)$, 
  we have
  \[ 1 \neq N:= \NKer_{\reals}(G) \leq  \R(G) < \erz{g}.
  \]
  The last inequality is strict since
   $\R(G)$ is normal in $G$, but $\erz{g}$ is not.
  In particular, the first claim of the lemma implies the second one.
  
  Let $\lambda\in \Lin\erz{g}$ be faithful.
  By Lemma~\ref{l:nt_id} applied to
  $N \nteq \erz{g}$ and since $N\neq 1$, it follows
  $\lambda(e_N)=0$.
  Thus
  \[ f = e_{\lambda} + e_{\overline{\lambda}} 
       = \frac{1}{ \card{\erz{g}} }
         \sum_{h\in \erz{g}} 
            \big( \overline{\lambda}(h) + \lambda(h) \big) h
       \in \reals \erz{g} 
  \]
  is an idempotent with $f e_N = 0$.
  It follows from Lemma~\ref{l:idemp}
  that $f$ is a central idempotent in 
  $\reals G$, and so
  $f^x = f$ for all $x\in G$.
  But by assumption, there is some $x\in G$ such that
  $g^x \notin \erz{g}$.
  It follows that $\overline{\lambda(g)} + \lambda(g) = 0$.
  As $\lambda(g)$ is an $n$-th root of unity,
  where $n=\ord(g)$, this is only possible when $\ord(g) = 4$.  
\end{proof}

By this lemma, $\card{ \NKer_{\reals}(G) } \leq 2$
except when $G$ is a Dedekind group.
Next we want to compute
$\NKer_{\reals}(G)$ in each case of 
Theorem~\ref{main:reals}.
In particular,
this will prove the ``if'' part of Theorem~\ref{main:reals}.

To show that certain characters are skew-linear,
we use the Frobenius-Schur indicator.
Recall that for $\chi\in \Irr G$, 
its \defemph{Frobenius-Schur indicator} is defined by
\[ \FS(\chi) := \frac{1}{\card{G}}
          \sum_{g\in G} \chi(g^2).
\]
When
$\FS(\chi) = 1$, then
$\chi = \overline{\chi}$ and $\chi$ is afforded by a 
representation with entries in $\reals$, 
so $m_{\reals}(\chi) = 1$.
When $\FS(\chi) = 0$, then
$\chi\neq \overline{\chi}$, and again $m_{\reals}(\chi)=1$.
Finally, when $\FS(\chi) = -1$, then
$\chi = \overline{\chi}$,
but $m_{\reals}(\chi) = 2$.
In the last case, 
there is a simple $\reals G$-module affording
$2\chi$, and
$\enmo_{\reals G}(S) \iso \quats$,
the division ring of Hamilton's quaternions
\cite[Theorem~13.12]{huppCT}.

In particular, 
$\chi\in \Irr G$ is skew-linear over $\reals$, 
if and only if either
$\chi(1)=1$ ($\chi$ is linear), or
$\chi(1) =  2$ and $\FS(\chi) = -1$.

When $G$ is abelian, then trivially
$\NKer_{\reals}(G)=G$.
Next we consider generalized dicyclic groups.
Notice that the Dedekind group $Q_8 \times (C_2)^r $
is generalized dicyclic, 
with abelian subgroup $ A \iso C_4 \times (C_2)^r$.

\begin{lemma}\label{l:gendickern}
  Let $G$ be generalized dicyclic, and
  let $g\in G$ and $A\nteq G$ be as in the definition
  (before Theorem~\ref{main:reals}).
  Then  $\R(G)= \NKer_{\reals}(G) = G$ if\/ 
        $G/ \erz{g^2}$ is abelian, 
        and\/ $\R(G) = \NKer_{\reals}(G) = \erz{g^2}$ 
        otherwise.
\end{lemma}
\begin{proof}
  First, observe that
  $g^2 = (g^2)^g = g^{-2}$ and thus
  $g^4 =1$. 
  Moreover, for any $a\in A$, we have
  $(ga)^2 = g^2 a^ga = g^2$.
  By assumption, $g^2\neq 1$.

  In view of Lemma~\ref{l:nn_ord4}, it suffices to show
  that $\erz{g^2} \subseteq \NKer_{\reals}(G)$, that is,
  all characters $\chi\in \Irr G$ with $g^2\notin \Ker \chi$
  are skew-linear.
  (In the case when $G/\erz{g^2}$ is abelian,
   all characters of $G/\erz{g^2}$
   are linear and thus it will follow that all characters of $G$
   are skew-linear and $G$ is Dedekind.
   Conversely, if $\NKer_{\reals}(G)> \erz{g^2}$,
   then $\NKer_{\reals}(G)=G$ by Lemma~\ref{l:nn_ord4},
   and then
   $G \iso Q_8 \times (C_2)^r$ by Theorem~\ref{t:local_dpdivrings}
   and $G/\erz{g^2}$ is abelian.)

  So suppose that $\chi \in  \Irr G$ is not linear, 
  and $g^2 \notin \Ker \chi$.
  Let $\lambda\in \Lin A$ be a constituent
  of the restriction $\chi_A$.
  Then  $\chi = \lambda^G$ 
  by Clifford theory~\cite[Corollary~6.19]{isaCTdov}.
  As $a^g = a^{-1}$ for all $a\in A$, we have
  $\lambda^g = \overline{\lambda}$. 
  Also, $\lambda(g^2) \neq 1$,
  and thus $\lambda(g^2) = -1$ and $\chi(g^2) = -2$.
  It follows that
  \begin{align*}
    \FS(\chi) 
      = \frac{1}{\card{G}}
                 \sum_{x\in G} \chi(x^2)
      &= \frac{1}{\card{G}}
          \sum_{a\in A} (\chi((ga)^2) + \chi(a^2))         
      \\
      &= \frac{1}{\card{G}}
         \left( -2\card{A} + 
                \sum_{a\in A} ( \lambda(a^2) + \overline{\lambda(a^2)} )
         \right)
     \\
     &= \frac{-2\card{A}}{\card{G}}
      = -1.
  \end{align*}
  Here we have used that $(ga)^2 = g^2$ for all $a\in A$,
  and that $\sum_{a\in A} \lambda(a^2) = \sum_{a\in A} \lambda^2(a) =0$
  since 
  $\overline{\lambda} \neq \lambda$ and thus $\lambda^2\neq 1$.
  Since $\FS(\chi)=-1$ and $\chi(1)=2$,
  it follows that $\chi$ is indeed skew-linear, as claimed.  
\end{proof}

\begin{lemma}\label{l:type3kernel}
  When $G = \erz{u} \times \erz{x,y} \times E$
  with $\erz{u} \iso C_4$, $\erz{x,y}\iso Q_8$ 
  and $E\iso (C_2)^r$,
  then $\NKer_{\reals}(G) = \R(G) = \erz{u^2x^2} \neq 1$.
\end{lemma}
\begin{proof}
  As $\erz{ux} \not\nteq G$, we have $\R(G) \leq \erz{u^2x^2}$.
  Let $\tau$ be the nonlinear irreducible character of $\erz{x,y}$ and 
  $\lambda$ a character of $\erz{u}$ with
  $\lambda \neq \overline{\lambda}$. 
  If $\chi$ is a character with
  $\chi(u^2x^2) \neq \chi(1)$, then either 
  $\chi$ is linear, or
  $\chi = \lambda^2 \times \tau \times \sigma$
  for some $\sigma\in \Lin E$.
  The latter characters all have $\FS(\chi)=-1$.
  Thus $ \erz{u^2 x^2} \subseteq \NKer_{\reals}(G) $.  
\end{proof}

\begin{lemma}
  When $G= \erz{u,v} \times \erz{x,y} \times E$
  with $\erz{u,v}\iso \erz{x,y} \iso Q_8$ and $E\iso(C_2)^r$,
  then $\NKer_{\reals}(G) = \R(G)= \erz{u^2x^2} \neq 1$.
\end{lemma}
\begin{proof}
  As $\erz{ux} \not\nteq G$, we have $\R(G) \leq \erz{u^2x^2}$.
  Let $\tau_1$ and $\tau_2$ be the nonlinear characters of
  $\erz{u,v}$ and $\erz{x,y}$, respectively.
  If $\chi(u^2x^2) \neq \chi(1)$, then either
  $\chi = \tau_1 \times \lambda \times \sigma$ with
  $\lambda \in \Lin\erz{x,y}$ and $\sigma\in \Lin(E)$,
  or $\chi = \lambda \times \tau_2 \times \sigma$ 
  with $\lambda\in \Lin\erz{u,v}$ and $\sigma\in \Lin(E)$.
  In both cases, $\FS(\chi)=-1$
  and thus $\erz{u^2x^2} \leq \NKer_{\reals}(G)$.
\end{proof}

This lemma finishes the proof of the ``if'' part of 
Theorem~\ref{main:reals}.
We now start working for the ``only if''-part.

\begin{lemma}\label{l:r_not2}
  Suppose that $1 < \NKer_{\reals}(G) < G$ and that
  $G$ is not a $2$-group.
  Then $G$ is generalized dicyclic.
\end{lemma}
\begin{proof}
  By Corollary~\ref{c:dede_nker} and Lemma~\ref{l:nn_ord4},
  we have that $\NKer_{\reals}(G) = \R(G) = \erz{z}$,
  where $z$ has order~$2$.
  Every odd-order subgroup of $G$ is normal in $G$,
  and in particular the Sylow $p$-subgroups,
  for $p$ odd, generate a normal $2$-complement, $U$,
  of $G$.
  As $U$ is Dedekind, it follows that $U$ is abelian.
  
  Now set $A= \C_G(U) $, which contains $U$.
  There is $g\in G$ such that $\erz{g}\not\nteq G$.
  By Lemma~\ref{l:nn_ord4}, we have $g^4 = 1$.
  If $gu=ug$ for some $u\in U$, 
  then $\erz{g}$ is characteristic in
  $\erz{gu} = \erz{g} \times \erz{u}$,
  and thus $\erz{gu}\not\nteq G$.
  Again by Lemma~\ref{l:nn_ord4},
  it follows that $(gu)^4 = 1$ and thus $u=1$.
  Thus $\C_U(g)=1$ and $g\notin A$.
  In particular, $A < G$.
  
  Conversely, let $g\notin A$, 
  and let $s=g_2$ be the $2$-part of~$g$.
  Then $gA = sA$ and thus $s\notin A$.
  Thus $u^s \neq u$ for some $u\in U$, 
  and thus $s^u = s[s,u] \notin \erz{s}$.
  It follows that $\erz{s}$ is not normal in $G$,
  and thus $\erz{g}$ is not normal in $G$.
  By Lemma~\ref{l:nn_ord4}, it follows that
  $g^4=1$ and $g^2 = z$ (and thus $s=g$).
  
  In particular, for $g\in G\setminus A$ and $a\in A$, we have
  $g^2 = z = (ga)^2 = g^2 a^g a$, 
  and thus $a^g = a^{-1}$.
  For $u\in U$ and $g$, $h\in G\setminus A$
  we have $u^g = u^{-1} = u^h$ and thus
  $gh^{-1}\in \C_G(U) = A$, so
  $\card{G:A} = 2$.
  Thus $G$ is generalized dicyclic.
\end{proof}

To finish the proof of the ``only if'' part
of Theorem~\ref{main:reals}, we use a part of
Blackburn's classification~\cite[Theorem~1]{Blackburn66}:

\begin{thm}[Blackburn 1966]\label{t:blackburn1}
  Let $G$ be a $p$-group with $\R(G)\neq 1$.
  Then one of the following holds:
  \begin{enumthm}
  \item $G$ is abelian.
  \item $p=2$ and $G$ is generalized dicyclic.
  \item $p=2$ and $G \iso C_4 \times Q_8 \times (C_2)^r$, $r\in \nats $.
  \item $p=2$ and $G \iso Q_8 \times Q_8 \times (C_2)^r$, $r\in \nats $.
  \end{enumthm}
\end{thm}

Using Theorem~\ref{t:blackburn1},
it is rather straightforward to determine all finite groups
$G$ with $\R(G)\neq 1$, but a rather long list 
emerges~\cite[Theorem~2]{Blackburn66}.
However, due to Lemma~\ref{l:r_not2},
 we do not need to go through the longer list of
finite groups with $\R(G)\neq 1$.

\begin{proof}[Proof of Theorem~\ref{main:reals}, ``only if'']
  Suppose that $\NKer_{\reals}(G)\neq 1$.
  If $\NKer_{\reals}(G) = G$, then
  $G$ is abelian or $G\iso Q_8 \times (C_2)^r$,
  by Theorem~\ref{t:local_dpdivrings}.
  
  If $1 < \NKer_{\reals}(G) <G$ and $G$ is not a $2$-group, 
  then $G$ is generalized dicyclic, by Lemma~\ref{l:r_not2}.
  If $G$ is a $2$-group, then it follows from
  Blackburn's classification of $2$-groups
  with $\R(G)\neq 1$ (Theorem~\ref{t:blackburn1}) 
  that $G$ appears on the list in Theorem~\ref{main:reals}.
\end{proof}

\section{Classification over the rational numbers}
\label{sec:classq}

In this section, we prove Theorem~\ref{main:qclass}.
Throughout, we write
$\NKer(G):=\NKer_{\rats}(G) $.

First, we show the ``if''-part of 
Theorem~\ref{main:qclass}.
We collect the necessary information in the following theorem:

\begin{thm}\label{t:qclass_if}
    Let $G$ be a finite group.
    \begin{enumthm}
    \item \label{it:mi_ab} 
       When $G$ is abelian, then $\NKer(G)=G$.
    \item \label{it:mi_nilp}
       Suppose $G = S \times A$,
       where $S$ is a $2$-group of exponent $4$
       which appears on the list from Theorem~\ref{main:reals},
       the group $A$ is abelian of odd order,
       and the multiplicative order of\/ $2$ modulo $\card{A}$ is odd.
       Then $1\neq  \NKer_{\reals}(S) \leq \NKer(G) $.
    \item \label{it:mi_gendic}%
       When $G$ is generalized dicyclic,
       then $\{1\}\neq \NKer_{\reals}(G) \leq \NKer_{\rats}(G)$.
    \item \label{it:mi_allab}
       Let $G =(PQ)\times B$ be as in 
      Theorem~\ref{main:qclass}~\ref{it:m_allab}.
      Then $1 \neq \NKer(G) \leq \erz{g^{p^c}} < \erz{g}$,
      with $g\in P$ and $c\geq 1$ as in Lemma~\ref{main:pq}.  
    \item \label{it:mi_dirprod_h}%
       Let $G = Q_8 \times (C_2)^r \times H$ 
       with $H$ a group of odd order,
       such that the multiplicative order of\/ $2$ in
       $(\ints/\card{H})^*$     
       is odd.
       Then $\NKer(H) \leq \NKer(G)$.
    \end{enumthm}
\end{thm}

We remark that in Case~\ref{it:mi_gendic},
we actually have $\NKer_{\reals}(G)=\NKer(G)$,
and in Case~\ref{it:mi_nilp},
we have $\NKer(G) =G$ when $\NKer_{\reals}(S)=S$,
and $\NKer(G) = \NKer_{\reals}(S)$ else.
This is not too difficult to see, but as we do not need this,
we omit the proofs.
In the situation of Theorem~\ref{t:qclass_if}\ref{it:mi_dirprod_h},
it follows from Lemma~\ref{l:subgroupskernel}
that $\NKer(G) = \NKer(H)$ whenever $\NKer(H) < H$.
In the situation of \ref{it:mi_allab},
suppose that $s\geq c$ is such that the factor group $G/\erz{g^{p^s}}$
still fulfills the various conditions in Theorem~D~\ref{it:m_allab}.
Then we must have $g^{p^{s-1}} \in \NKer(G)$ by \ref{it:mi_allab} above.
In fact, one can show that $\NKer(G)= \erz{g^{p^{s-1}}}$,
where $s$ is the smallest number such that
$G/\erz{g^{p^s}}$ is as in Theorem~\ref{main:qclass}~\ref{it:m_allab}.

Notice that Theorem~\ref{t:qclass_if}~\ref{it:mi_allab}
contains Lemma~\ref{main:pq} from the introduction.
Theorem~\ref{t:qclass_if} shows in particular that
$ \NKer(G) \neq 1 $ 
for all the groups occurring in Theorem~\ref{main:qclass}.

\begin{proof}[Proof of 
     Theorem~\ref{t:qclass_if},
     \ref{it:mi_ab}--\ref{it:mi_gendic} and \ref{it:mi_dirprod_h}]
 Part~\ref{it:mi_ab} is trivial,
 and \ref{it:mi_gendic} follows from
   Lemma~\ref{l:gendickern}
   and Lemma~\ref{l:subfields}.

   Now assume the situation of \ref{it:mi_nilp}.
   By Theorem~\ref{main:reals},
   we have $\NKer_{\reals}(S)\neq 1$.
   Pick 
   $\chi\in \Irr(G)$ with
   $\NKer_{\reals}(S) \not\subseteq \Ker(\chi)$.
   We have to show that $\chi$ is skew-linear over $\rats$.
   We can write $\chi= \tau \times \lambda$
   with $\tau\in \Irr(S)$ and $\lambda\in \Lin(A)$.
   Since $\NKer_{\reals}(S)\not\subseteq \Ker(\tau)$,
   we either have $\tau(1)=1$,
   or $\tau(1)=2$ and $\FS(\tau)=-1$.
   When $\tau(1)=1$, then $\chi$ is linear.
   In the second case, 
   $\reals S e_{\tau}$ is isomorphic to the quaternions over the reals.
   Since $S$ has exponent~$4$,
   we have that $\rats(\tau) = \rats$ 
   and $\rats S e_{\tau} \iso \quats$,
   the rational quaternions.
   Thus $\rats G e_{\chi} \iso \quats \otimes_{\rats} \rats(\lambda)$.
   By assumption and by Lemma~\ref{l:split_quat}\ref{it:qu_cyc},
   the latter is a division ring.
   Thus $\chi$ is skew-linear over $\rats$ as required.   
   
   Finally, let $G$ be as in \ref{it:mi_dirprod_h},
   and write $S= Q_8 \times (C_2)^r$.
      Let $\chi\in \Irr(G)$ be such that
      $\NKer(H) \not\subseteq \Ker(\chi)$.
      We have to show that $\chi$ is skew-linear over $\rats$.
      We can write $\chi = \sigma\times \theta$
      with $\sigma\in \Irr(S)$ and $\theta\in \Irr(H)$.
      As $\NKer(H)\not\subseteq \Ker(\theta)$,
      the character~$\theta$ is skew-linear over $\rats$.
        Let $D$ be the block ideal of $\rats H$ corresponding to $\theta$.
        This is a division ring with center isomorphic to $\rats(\theta)$.
        If $\sigma$ is linear,
        then the block ideal corresponding to 
        $\chi = \sigma \times \theta$
        is again isomorphic to $D$.
        If $\sigma$ is nonlinear,
        then the block ideal  corresponding to $\chi$
        is isomorphic to
        \[ (\quats \otimes_{\rats} \rats(\theta))
            \otimes_{\rats(\theta)}
            D,
        \]
        by Lemma~\ref{l:blocksdirprod}.
        The first factor is a division ring by 
        Lemma~\ref{l:split_quat}\ref{it:qu_cyc}.
        Therefore, the tensor product is also a division ring,
        since both factors are division rings,
        and the first has dimension $4$ over its center,
        while $D$ has odd dimension~$\theta(1)^2$ over its center.
        Thus $\chi$ is skew-linear.
        This shows that
        $\NKer(H) \leq \NKer(G)$
        as required.        
\end{proof}

For the proof of Theorem~\ref{t:qclass_if}~\ref{it:mi_allab},
we need some lemmas first.
Recall that $m_q(\chi)$ denotes the Schur index
over $\rats_q$, the field of $q$-adic numbers.

\begin{lemma}\label{l:localsi_abbg}
  Let $G$ be a finite group and $q$ a prime.
  Suppose that $G$ 
  has a normal abelian Sylow $q$-subgroup $Q$,
  such that every subgroup of $Q$ is normal in $G$.
  Let $H$ be a $q$-complement in $G$.
  If $\chi\in \Irr(G)$,
  then 
  $m_q(\chi) = \card{\rats_q(\chi,\theta) : \rats_q(\chi)}$
  for any irreducible constituent $\theta\in \Irr(H)$
  of $\chi_H$.
\end{lemma}

   Notice that the $q$-complement $H$ exists by 
   the Schur-Zassenhaus theorem.
   We allow $Q=1$ and $H=G$ in this lemma.
   
\begin{proof}[Proof of Lemma~\ref{l:localsi_abbg}]  
  Let $\lambda\in \Lin(Q)$ be a constituent of $\chi_Q$.
  Then $K=\Ker(\lambda)$ is normal in $G$
  by assumption and thus
  $K\subseteq \Ker(\chi)$.
  We may factor out $K$ and assume without loss of generality
  that $K=1$.
  
  This means that $Q$ is cyclic and thus $\chi$ is in a 
  $q$-block with cyclic defect group.
  Thus we can apply Benard's theorem~\cite{Benard76}
   to $\chi$
   and conclude that
  $m_q(\chi) = \card{\rats_q(\chi,\phi):\rats_q(\chi)}$
  for any irreducible Brauer constituent $\phi$ of $\chi$.
  But an irreducible Brauer character of $G$ contains 
  the normal $q$-subgroup
  $Q$ in its kernel, and thus can be identified
  with an ordinary character of the $q'$-group 
  $H \iso G/Q$.
  Thus if $\phi$ is an irreducible Brauer constituent of $\chi$,
  then $\phi_H=\theta\in \Irr(H)$ is an
  irreducible constituent of $\chi$,
  and the result follows from Benard's theorem.
\end{proof}

The next observation was already used in Blackburn's classification.

\begin{lemma}\label{l:specialauts}
  Let $Q$ be a finite abelian $q$-group ($q$ prime)
  and let $P$ be some finite group with
  $(\card{P}, \card{Q} ) = 1$.
  Suppose that $P$ acts on $Q$ by automorphisms
  such that
  every subgroup of $Q$ is $P$-invariant.
  Then $P/\C_P(Q)$ is cyclic of order dividing $q-1$,
  and $\C_P(x) = \C_P(Q)= P_{\lambda}$ for every $1\neq x\in Q$
  and $1_Q\neq \lambda \in \Lin Q$.
\end{lemma}
\begin{proof}
  Take $x\in Q$ of maximal order and $u\in P$.
  Since $x^u \in \erz{x}$ by assumption,
  we have $x^u = x^k$ for some $k \in \nats$.
  If $y\in Q$ with $\erz{x}\cap \erz{y} = 1$,
  then $y^u = y^k$, since $u$ maps 
  $\erz{y}$ and $\erz{xy}$ to itself.
  It follows that $y^u = y^k $ for all $y\in Q$.
  
  Therefore, $P/\C_P(Q)$ is isomorphic to a $q'$-subgroup
  of $\Aut(\erz{x})$, and thus is cyclic of order dividing
  $q-1$.
  
  Finally, suppose $1 \neq x\in Q$ and
  $x^u = x$ for some $u\in P$.
  As we have just seen, 
  there is $k\in \nats$ such that
  $y^u = y^k$ for all $y\in Q$.
  It follows that $k\equiv 1 \mod q$
  (as $q \mid \ord(x)$). 
  Since $\card{P/\C_P(Q)}$ divides $q-1$, 
  it follows that
  $k^{q-1} \equiv 1 \mod q^n$,
  where $q^n$ is the exponent of $Q$.
  But this yields that $k\equiv 1 \mod q^n$ and thus
  $u\in \C_P(Q)$ as claimed.
  The proof for $\lambda\in \Lin Q$ is similar,
  using that there is $\ell$ such that
  $\mu^u = \mu^{\ell}$ for all $\mu \in \Lin Q$.
\end{proof}

Next, we compute some $q$-adic Schur indices for a class of groups
containing those in Theorem~\ref{main:qclass}~\ref{it:m_allab}.
This result will also be used later in the proof of the
``only if''-part.

\begin{lemma}\label{l:qsi_allab}
  Let $p\neq q$ be primes, let
  $P$ be an abelian $p$-group
  and $Q$ an abelian $q$-group.
  Suppose $P$ acts on $Q$ by automorphisms such that
  every subgroup of $Q$ is $P$-invariant, 
  and set $C= \C_P(Q)$.
  Let $B$ be an abelian group with 
  $(pq, \card{B})=1$
  and let $G=(PQ)\times B$.  
  Then any nonlinear $\chi\in \Irr(G)$ has the form
  $\chi = (\mu \times \lambda)^G$ for some
  $\mu\in \Lin(CB)$ and $1\neq \lambda\in \Lin(Q)$.
  Let $\theta\in \Lin( PB \mid \mu)$. 
  Then $m_q(\chi) = \ell/k$,
  where $\ell$ is the smallest positive integer
  such that $\ord(\theta)$ divides $q^{\ell}-1$,
  and $k$ is the smallest
    positive integer such that $\ord(\mu)$ divides $q^k-1$.
\end{lemma}

(In other words, $\ell$ and $k$ are the multiplicative orders
of $q$ modulo $\ord(\theta)$ and modulo $\ord(\mu)$, 
respectively.)

\begin{proof}[Proof of Lemma~\ref{l:qsi_allab}]
  Notice that $CB \leq \Z(G)$, and that
  $H = PB$ is an abelian $q$-complement of $G$.
  We are in the situation of Lemma~\ref{l:localsi_abbg}.
  Let $\chi\in \Irr(G)$.  
  If $Q\leq \Ker(\chi)$, then $\chi$ is linear,
  since $G/Q = P\times B$ is abelian.
  (In fact, $Q=G'$ when $C<P$.)
  Otherwise, let $\lambda \neq 1$ be a linear constituent of $\chi_Q$.
  By Lemma~\ref{l:specialauts},
  we have $G_{\lambda} = CBQ $.
  By Clifford theory,
  $\chi$ is induced from some linear character of 
  the abelian group $CBQ$,
  say $\chi = (\mu \times \lambda)^G$ with
  $\mu \in \Lin(CB)$.
  It follows that $\rats_q(\chi) = \crp{K}(\mu)$, where
  $\crp{K}\subseteq \rats_q(\lambda)$ is totally ramified over
  $\rats_q$,
  and the extension $\crp{K}(\mu)/\crp{K}$ is unramified.
  By the general form of unramified extensions,
  the residue field of $\rats_q(\chi)=\crp{K}(\mu)$ has order $q^k$, 
  where $k$ is the smallest
  positive integer such that $\ord(\mu)$ divides $q^k-1$.
  
  The restriction $\chi_H$ to the $q$-complement $H=PB$
  is the sum of all linear characters
  $\theta \in \Lin(H)$ lying over $\mu$.
  Thus $\rats_q(\chi,\theta)= \crp{K}(\theta)$ 
  is generated by $\rats_q(\chi)$ and 
  a root of unity of order $\ord(\theta)$.
  Since $\ord(\theta)$ is not divisible by $q$,
  the extensions
  $\rats_q(\chi,\theta)/\rats_q(\chi)$ and 
  $\crp{K}(\theta)/\crp{K}$ are unramified.
  We can thus compute 
  $\card{\rats_q(\chi,\theta) : \rats_q(\chi)}$
  by computing the degrees of the residue fields.
  As above, the residue field of 
  $\rats_q(\chi,\theta) = \crp{K}(\theta)$ has order
  $q^{\ell}$, where $\ell$ is the smallest positive integer
  such that $\ord{\theta}$ divides $q^{\ell}-1$.
  Now the result follows from Lemma~\ref{l:localsi_abbg}.
\end{proof}

\begin{lemma}\label{l:allab_suff}
  Let $G =(PQ)\times B$ be as in 
  Theorem~\ref{main:qclass}~\ref{it:m_allab}.
  Then $\NKer_{\rats_q}(G)\cap \erz{g}\neq 1$,
  with $g\in P$ as in Lemma~\ref{main:pq}.  
\end{lemma}
\begin{proof}
  Recall that $P= \erz{g}\times P_0$,
  where $g$ has order $p^{c+d}$
  and $C:=\C_P(Q)= \erz{g^{p^c}}\times P_0$.
  Let $z\in \erz{g^{p^c}}$ be an element of order $p$.
  We claim that $z\in \NKer_{\rats_q}(G)$.
  
  Suppose that $\chi(z)\neq\chi(1)$
  for $\chi\in \Irr G$.
  If $\chi(1)> 1$, then
  $\chi = (\mu \times \lambda)^G$ as in 
  Lemma~\ref{l:qsi_allab}, with
  $\mu\in \Lin(CB)$ and $1\neq \lambda\in \Lin Q$.
  
  To compute the Schur index of such a $\chi$,
  we apply Lemma~\ref{l:qsi_allab}.
  Since $\mu(z)\neq 1$, it follows that
  $\ord(\mu) = p^d n$, where $n$ divides $\card{B}$.
  For $\theta \in \Lin(PB\mid \mu)$, we have
  $\ord(\theta) = p^c \ord(\mu) = p^{c+d}n$.
  Let $k$ and $\ell$ be the multiplicative order of
  $q$ modulo $p^d n$ and $p^{c+d}n$, respectively.
  We have to show that $\ell/k = p^c$.
  
  Let $p^e = (q-1)_p$ and write $q = 1 + p^ex$ with $(p,x)=1$.
  Then $q^p \equiv 1 + p^{e+1}x \not\equiv 1\mod p^{e+2}$
  except when $p^e = 2$.
  By the assumptions in Lemma~\ref{main:pq}, 
  $p\leq p^c \leq p^e \leq p^d$.
  It follows that 
  $(q^{k_0} -1)_p = p^d$
  for $k_0 = p^{d-e}$ when $p^e>2$.
  When $p^e=2$, the extra condition in Lemma~\ref{main:pq}
  also ensures that 
  $(q^{k_0}-1)_p = p^d$,
  when $k_0$ is the multiplicative order of $q$ modulo $p^d$.
  Additionally,
  the assumption  in Theorem~\ref{main:qclass}~\ref{it:m_allab}
  on $\card{B}$
  yields that $k/k_0$ is not divisible by $p$.
  Thus $q^k = 1 + p^dx$ for some $x\in \ints$
  with $(p,x)=1$, and $p^d>2$.
  It follows $\ell = p^c k$
  and thus $m_q(\chi) = p^c = \chi(1)$.
  This was to be shown.
\end{proof}

\begin{proof}[Proof of Theorem~\ref{t:qclass_if}~\ref{it:mi_allab}]
   As $\NKer_{\rats_q}(G) \leq \NKer(G)$ for any group,
   we see that $\NKer(G)\cap \erz{g} \neq 1$.
   On the other hand, we have that
   $\R(G) \leq \erz{g^{p^c}}$
   (because $ \erz{ g^{ p^{c-1} } } $ is not normal in $PQ$),
   and thus $\NKer(G) \leq \erz{g^{p^c}}$
   by Lemma~\ref{l:many_normals}.   
\end{proof}

This finishes the proof of the ``if''-part of
Theorem~\ref{main:qclass}.

Next, we prove Theorem~\ref{main:qclass}
for nilpotent groups.

\begin{thm}\label{t:qclass_nilp}
  Let $G$ be nilpotent. 
  Then 
  $\NKer(G)\neq 1$
  if and only if
  one of the following holds.
  \begin{enumthm}
  \item $G$ is abelian and $G \neq \{1\}$.
  \item \label{it:mn_nilp}%
        $G = S \times A$,
        where $S\in \Syl_2(G)$ is a nonabelian group
        from the list in Theorem~\ref{t:blackburn1}
        and has exponent $4$,
        and $A$ is abelian of odd order,
        and the multiplicative order of $2$ in 
        $(\ints/\card{A})^*$ is odd.
  \item \label{it:mn_gendic}%
        $G$ is a generalized dicyclic $2$-group.
  \end{enumthm} 
\end{thm}
\begin{proof}
  The ``if''-part has been shown in Theorem~\ref{t:qclass_if},
  so assume $\NKer(G)\neq 1$.
  Then $1\neq \NKer(P) \leq \R(P)$ 
  for every nontrivial subgroup~$P \leq G$.
  By Theorem~\ref{t:blackburn1},
  the Sylow subgroups of odd order are all abelian.
  Thus we have $G= S\times A$ with $S\in \Syl_2(G)$ and
  $A$ abelian of odd order,
  and $S$ appears on the list from Theorem~\ref{t:blackburn1}.
  We may assume that $S$ is nonabelian.
  
  If $S$ has exponent~$4$, 
  then the nonlinear, but skew-linear characters
  of $S$ yield the quaternions over $\rats$ as block ideal 
  of the rational group algebra $\rats S$,
  and it follows from 
  Lemma~\ref{l:split_quat}\ref{it:qu_cyc}
  and Lemma~\ref{l:blocksdirprod}
  that \ref{it:mn_nilp} holds.
  
  If $S$ contains elements of order $8$ or greater,
  then $S$ is generalized dicyclic,
  and $\NKer(S) < S$ is cyclic of $2$-power order.
  It remains to show that $A=1$ in this case.
  Let $z\in \NKer(S)$ be the element of order~$2$.
  There is a $\sigma\in \Irr S$ such that
  $S/\Ker(\sigma) $ is a dicyclic (=generalized quaternion)
  group of order at least $16$,
  and $z\not\in \Ker(\sigma)$.
  Then $\rats(\sigma)$ contains $\sqrt{2}$.
  As $S/\Ker(\sigma)$ has a subgroup of order~$8$
  isomorphic to the quaternion group,
  the block ideal of $\rats S$ corresponding to $\sigma$  
  is isomorphic to the quaternions over
  a field containing $\sqrt{2}$.

  Let $\lambda \in \Lin A$
  and $\chi = \sigma \times \lambda$.
  The block ideal corresponding to $\chi$ is isomorphic to
  \[  (\quats \otimes_{\rats} \rats(\sigma))\otimes_{\rats(\sigma)}
          \rats(\sigma,\lambda)
      \iso \quats \otimes \rats(\sigma,\lambda),
  \] 
  by Lemma~\ref{l:blocksdirprod}.
   By Lemma~\ref{l:split_quat}\ref{it:qu_qu},
  $\quats \otimes \rats(\sqrt{2},\lambda)$ 
  can be a division ring only when $\lambda=1_A$.
  Since $\lambda \in \Lin A$ was arbitrary,
  we have $A=1$ as required.
\end{proof}

The proof of Theorem~\ref{main:qclass} (``only if'') for 
nonnilpotent groups will be divided into a number of lemmas.
Recall that a \defemph{Blackburn group} 
is a finite group $G$ such that $\R(G)\neq 1$.
By Lemma~\ref{l:many_normals},
a group with $\NKer(G)\neq 1$ is a Blackburn group.
For later reference, we record the following observation
(which is part of the argument used by Blackburn to classify
these groups):

\begin{lemma}\label{l:bb_pnilp}
  Let $G$ be a Blackburn group and 
  $p$ a prime dividing $\card{\R(G)}$.
  Then $G$ has a normal $p$-complement $A$
  such that every subgroup of $A$ is normal in $G$.
  If $A$ is nonabelian, then 
  $G = Q_8 \times (C_2)^r \times H$,
  where $H$ is a Blackburn group of odd order.
\end{lemma}
\begin{proof}
  By definition of $\R(G)$, all the Sylow $q$-subgroups for
  $q\neq p$ are normal in $G$,
  and thus generate a normal $p$-complement, $A$.
  By definition of $\R(G)$, it follows also that every
  subgroup of $A$ is normal in $G$.
  
  In particular, $A$ is a Dedekind group.
  If $A$ is nonabelian, then
  $S\in \Syl_2(G)$ is isomorphic to
  $Q_8\times (C_2)^r$, by Theorem~\ref{t:dedekind}.
  As $S\nteq G$, there is a $2$-complement $H$.
  Since every subgroup of $S$ is normal in $G$,
  it is easy to see that $H$ centralizes $S$
  and thus $G = S \times H$
  (this is also shown 
  in~\cite[Proof of Theorem~2(e)]{Blackburn66}).
  Any nonnormal subgroup of $H$ is nonnormal in $G$
  and thus $ 1 \neq \R(G)\leq \R(H)$.
\end{proof}

Thus $A$ is a Dedekind group
and $G = PA$ for any $P\in \Syl_p(G)$.
The classification of Blackburn groups can now be obtained
by considering the different possibilities for $P$ and $A$
(using the fact that $P$ is also a Blackburn group and
 Theorem~\ref{t:blackburn1} for $P$, 
 and Theorem~\ref{t:dedekind} for $A$).
However, in our proof of Theorem~\ref{main:qclass},
we do not have to consider all the cases of Blackburn's
classification separately.
The next result will be used
to reduce to the case that
$A$ is abelian.

\begin{lemma}\label{l:qclass_q8c2h}  
  Let $G = Q_8 \times (C_2)^r \times H$ with
  $H \neq 1$ of odd order,
  and suppose $\NKer(G)\neq 1$.
  Then $\NKer(H)\neq 1$ and 
  the multiplicative order of\/ $2$ in $(\ints/\card{H})^*$
  is odd.
\end{lemma}
\begin{proof}
  When $\NKer(H) =H$ 
  then $H$ is abelian and $G$ is a Dedekind group,
  and Corollary~\ref{c:dede_nker} and Theorem~\ref{t:qg_divringprod}
  yield the result.
  When $\NKer(H)  < H$, then 
  $1\neq \NKer(G) \leq \NKer(H) $
  by Lemma~\ref{l:subgroupskernel}.

  Let $z\in \NKer(G)$
  have prime order $p$, 
  and let $A$ be the abelian $p$-complement of $H$. 
  Then $\erz{z}\nteq G$ and so $\erz{z,A}=\erz{z}\times A$.
  Suppose $\lambda\in \Lin(\erz{z,A})$ has maximal possible order.
  (This implies $\lambda(z)\neq 1$, in particular.)
  Then any $\chi\in \Irr(H\mid \lambda)$ is skew-linear.
  For
  $\tau\in \Irr(Q_8)$ with $\tau(1)=2$,
  we must have that $\tau \times \chi$ is also skew-linear.
  Lemma~\ref{l:blocksdirprod} yields in particular,
  that $\rats(\chi)$ must not be a splitting field for $\quats$
  (the quaternions over $\rats$).

  Cleary,  $\rats(\chi)$, $ \rats(\lambda) \subseteq \rats(\eps)$,
  where $\eps$ is a primitive $\card{H}$-th root of unity.
  Since every prime dividing $\card{H}$ also divides
  $\ord(\lambda)$,
  we see that
  $\card{\rats(\eps):\rats(\lambda)}$ is odd.
  By Lemma~\ref{l:res_skewlin}, all constituents of $\chi_A$
  are Galois conjugate with $\lambda$ over $\rats(\chi)$.
  Thus $\card{\rats(\chi,\lambda):\rats(\chi)}$ divides $\chi(1)$
  and is odd. 
  It follows that $\card{\rats(\eps):\rats(\chi)}$ is odd as well.
  Thus $\rats(\chi)$ is a splitting field for the quaternions,
  if and only if $\rats(\eps)$ is a splitting field 
  for the quaternions.
  Now Lemma~\ref{l:split_quat}\ref{it:qu_cyc} yields that the condition
  on the order of $2 \mod \card{H}$ holds.
\end{proof}

\begin{lemma}\label{l:abbglocal}
  Let $G =PA$ be a Blackburn group with 
  normal abelian $p$-comple\-ment~$A$ and\/
  $\R(G) \leq P\in \Syl_p(G)$.
  Suppose that 
  $\chi\in \Irr(G)$ is skew-linear over $\rats_q$,
  where $q$ is a prime dividing $\card{A}$.
  Then $P$ centralizes every
  Sylow $r$-subgroup $R$ of $A$
  such that $r\neq q$ and $R\not\subseteq \Ker(\chi)$.
\end{lemma}
\begin{proof}
  Let $r\neq q$ and $R\in \Syl_r(A)$,
  and assume that $R\not\subseteq \Ker(\chi)$.
  Let $\lambda\in \Lin(R)$ be a linear constituent
  of $\chi_R$, so that $\lambda\neq 1$.
  
  By Lemma~\ref{l:specialauts},    
  $\C_P(R) = \C_P(x)$ for any $1\neq x\in R$, and thus also
  \[\C_P(R) = P_{\lambda} := \{u\in P \mid \lambda^u = \lambda\}.
  \]
  
  Consider the subgroup $H=PR$, and choose
  a constituent $\theta\in \Irr(H)$ of $\chi_H$
  that lies over $\lambda$.
  Then $\theta = \psi^H$ for some
  $\psi\in \Irr(H_{\lambda})$,
  where $H_{\lambda} = \C_P(R) R$.
  Thus $\theta(1) \geq \card{P:\C_P(R)}$.  
  On the other hand, by Corollary~\ref{c:localindequ} we have
  that $\theta(1)=1$, and thus 
  $P=\C_P(R)$ as claimed.
\end{proof}

\begin{lemma}\label{l:abab_1}
  Let $G = PA$ be a finite group,
  where $A$ is a normal abelian $p$\nbd complement
  and $1 \neq \NKer(G)\leq P\in \Syl_p(G)$.
  Suppose that 
  $\card{P:\C_P(A)} > 2$.
  Then $P$ is abelian,
  and there is exactly one Sylow subgroup
  of $A$ which is not centralized by $P$.
\end{lemma}
\begin{proof}  
  Let $z\in \NKer(G) \subseteq P$ be an element of order $p$.
  Choose $\tau\in \Irr(P)$ with
  $z\notin \Ker(\tau)$.  
  For $\lambda\in \Lin(A)$ arbitrary, we have
  \[ \ipcf{ (\tau^G)_A }{ \lambda }_A
        = \ipcf{ (\tau_{P\cap A})^A }{ \lambda }_A
        = \ipcf{ \tau_{P\cap A} }{ \lambda_{P\cap A} }_{P\cap A}
        = \tau(1) > 0,
  \]
  as $P\cap A = 1$.
  Thus for any $\lambda \in \Lin(A)$, there is 
  $\chi\in \Irr(G)$ 
  such that $\ipcf{\chi}{\tau^G}> 0$ and
  $\ipcf{\chi_A}{\lambda}> 0$.
  We apply this to a $\lambda$ such that
  $\lambda_R \neq 1_R$ for each Sylow subgroup, $R$, of $A$.
  Thus there is a $\chi\in \Irr(G)$ lying over $\tau$
  and such that $\Ker(\chi)$
  contains no Sylow subgroup of $A$.  

  Notice that $\C_P(A) = P_{\lambda}$ for such a $\lambda$,
  by Lemma~\ref{l:specialauts}.
    As $\chi$ is induced from a character of $G_{\lambda}$,
    it follows that 
    $\chi(1) \geq \card{G:G_{\lambda}} = \card{P:\C_P(A)} > 2$.

  Because $z\not \in \Ker(\chi)$ and $z\in \NKer(G)$,
  it follows that $\chi$ is skew-linear over $\rats$
  and thus $m_{\rats}(\chi) = \chi(1)$. 
  By Ito's theorem~\cite[Theorem~6.15]{isaCTdov},
  $\chi(1)$ divides $\card{G:A}=\card{P}$
  and thus is a power of $p$.  
  It follows from Lemma~\ref{l:sifacts}\ref{it:globloc}
  that there is a prime $q$ (possibly infinite) 
  such that $m_q(\chi)= m_{\rats}(\chi) = \chi(1)$. 
  
  Since $\chi(1)\geq \card{P:\C_P(A)} > 2$,
  it follows from Lemma~\ref{l:sifacts}\ref{it:localsi_div}
  that the prime $q$, such that 
  $m_q(\chi)=\chi(1)$, must be a finite, odd prime 
  and $q\neq p$.
  It follows that $q$ divides $\card{A}$.
  Now Corollary~\ref{c:localindequ} yields that
  $\tau$ is linear.
  Since the only assumption on $\tau\in \Irr P$ was 
  that $z\not\in \Ker(\tau)$,
  Lemma~\ref{l:kernelnonlin} yields that $P$ is abelian.
  Lemma~\ref{l:abbglocal} yields that
  $P$ centralizes all Sylow subgroups of $A$
  except the Sylow $q$-subgroup. 
\end{proof}

\begin{lemma}\label{l:c2ab}
  Let $G = SA$ be a finite group,
  where $A$ is a normal (abelian) $2$\nbd complement 
  and $1 \neq \NKer(G) \leq S\in \Syl_2(G)$, and suppose
  $\card{S:\C_S(A)} = 2$.
  Then either $G$ is as in Lemma~\ref{l:abab_1}
  (with $p=2$ and $P=S$),
  or $G$ is generalized dicyclic.
\end{lemma}

(Notice that $A$ is automatically abelian here since
 $A$ is a Dedekind group of odd order.)
 
\begin{proof}[Proof of Lemma~\ref{l:c2ab}]
  First we show that $C:= \C_S(A)$ is abelian.
  Let $z\in \NKer(G)$ have order $2$.
  Then $z\in \Z(G)$ and thus $z\in C$.
  If $C$ is not abelian, there is $\tau\in \Irr(C)$
  with $\tau(z) \neq \tau(1) > 1$ (Lemma~\ref{l:kernelnonlin}).
  Let $t\in S\setminus C$, and
  let $\lambda\in \Lin(A)$ be such that
  $\lambda^t \neq \lambda$.
  Then $\chi:= (\tau \times \lambda)^G \in \Irr(G)$,
  and $\chi(1) = 2\tau(1) > 2$.
  By Lemma~\ref{l:sifacts}~\ref{it:localsi_div},
  $\chi$ can not be skew-linear over $\reals$
  or $\rats_2$.
  Since $\chi_C$ has a non-linear constituent, 
  $\chi$ can not be skew-linear over $\rats_q$ 
  for odd primes $q$, by Corollary~\ref{c:localindequ}.
  As $\chi(1) = 2^r$, it follows from 
  Lemma~\ref{l:sifacts}~\ref{it:globloc} that
  $m_{\rats}(\chi) <  \chi(1)$,
  and thus
  $z\notin \NKer(G)$, contradiction.
  Thus $C$ is abelian as claimed,
  and $G=SA$ has the abelian subgroup 
  $ CA = C \times A $ 
  of index~$2$.
  
  Fix $t\in S\setminus C$.
    Notice that 
    $A = [A,t] \times C_A(t)$.
    Since every subgroup of $A$ is normal in $G$, 
    the factors of this decomposition have coprime orders.
  Also, we have $[A,t] \neq 1$ by assumption,
  and $t$ inverts the elements in $[A,t]$.
  
  Consider first the case $\C_A(t)\neq 1$.
  Pick some $\lambda \in \Lin(A)$ such that 
  $\Ker(\lambda)$ contains no Sylow subgroup of $A$.
  Then $\lambda^t \notin \{\lambda, \overline{\lambda}\}$.
  Consider some $\nu \in \Lin(C)$ with $\nu(z) =-1$,
  where $z\in \NKer(G)$ has order $2$ as before,
  and set $\mu = \nu \times \lambda \in \Lin(CA)$.
  As $\mu^t \neq \mu$,
  we have $\chi = \mu^G \in \Irr(G)$.
  This $\chi$ remains irreducible modulo~$2$, and thus
  $m_2(\chi)=1$, by Lemma~\ref{l:sifacts}~\ref{it:localsi_brauer}.
  As $\mu^t \neq \overline{\mu}$, we have also
  $m_{\reals}(\chi)=1$.
  But as $z\notin \Ker(\chi)$,
  it follows that $m_q(\chi)=2$ for some odd prime
  $q$ dividing $\card{A}$. 
  Then Lemma~\ref{l:abbglocal} yields
  that $S$ centralizes every Sylow subgroup of
  $A$ except one.
  Also Corollary~\ref{c:localindequ}
  yields that $\chi_S$ is a sum of linear characters.
  It follows that $\nu^t = \nu$
  for all $\nu\in \Lin(C)$ with $\nu(z)=-1$.
  Thus $S$ is abelian and $G$ is as in Lemma~\ref{l:abab_1}
  with $p=2$ in this case.
  
  Now assume that $\C_A(t) = 1$.
  If $S$ is abelian and $C$ is not just 
  an elementary abelian $2$-group,
  then again we find $\mu \in \Lin(CA)$ 
  with $\mu^t \notin \{\mu, \overline{\mu}\} $
  and $\chi = \mu^G$ as above,
  so that $m_q(\chi) = 2$ for some odd prime $q$,
  and the result follows again.
  
  If $S$ is abelian and $C$ is elementary $2$-abelian,
  then $G = S [A,t]$ is generalized dicyclic.
  
  Finally, assume that $\C_A(t) =1$
  and that $S$ is nonabelian.
  The $2$-group~$S$ occurs on the list in Theorem~\ref{t:blackburn1}.
  If $S\iso C_4 \times Q_8 \times (C_2)^r$
    or $S\iso Q_8 \times Q_8 \times (C_2)^r$, 
    then $S$ is generated by elements $u$ such that
    $\erz{u}\cap \R(S) =1$ and thus $S$ would centralize
    $A$, so this is impossible.
  (Alternatively, look at Blackburn's 
    list~\cite[Theorem~2]{Blackburn66}.)
  Thus $S$ is generalized dicyclic,
  and $S$ has an abelian subgroup $D$ of index $2$,
  such that $d^s=d^{-1}$ for all $d\in D$ and $s\in S\setminus D$.
  If $D=C$, then $G$ is generalized dicyclic.
  Thus we may assume that $D\neq C$.
  If $S$ is Dedekind, then $S\iso Q_8 \times (C_2)^r$,
  and we could choose $D=C$.
  So we can assume that $S$ is not Dedekind,
  and thus $\R(S) = \erz{z}$ by Lemma~\ref{l:gendickern}.
  We may choose $t\in D\setminus C$ and
  $s\in C\setminus D$.
  Since both $C$ and $D$ are abelian,
  it follows that
  $s$ centralizes $C\cap D$,
  and at the same time inverts the elements in $C\cap D$.
  Thus $C\cap D$ has exponent~$2$.
  Since $\card{S:C}=\card{S:D}=2$ and
  $z \in \erz{s}\cap\erz{t}$,
  it follows $s^2 = t^2 = z$.
  Since $st\notin D$, we also have $(st)^2 = z$.
  It follows that $\erz{s,t}\iso Q_8$ and
  $S = \erz{s,t} \times (C\cap D)
     \iso Q_8 \times (C_2)^r$.
  But then $S$ is Dedekind and $G$ generalized dicyclic, 
  contradiction.
\end{proof}

Finally, we recall a part of~\cite[Theorem~2(a)]{Blackburn66}.

\begin{lemma}\label{l:blackburn2a}
  Let $G =PA$ be a nonabelian Blackburn group such that
  $\R(G) \leq P\in \Syl_p(G)$ and such that $P$
  and the normal $p$-complement $A$ are abelian. 
  Then we can write 
      $ P = \erz{g} \times P_0$, such that
      $\C_P(A) = \erz{g^{p^c}} \times P_0$ ($c\geq 1$),
      and $p^d := \ord( g^{p^c} )$ is the exponent
      of $\C_P(A)$.
   There is a $k \in \nats$ such that
      $a^g = a^k$ for all $a\in A$.
\end{lemma}

\begin{proof}[Proof of Theorem~\ref{main:qclass}]
  Assume that $\NKer(G)\neq 1$ and let $p$ be a prime
  dividing $\card{\NKer(G)}$.
  By Lemma~\ref{l:bb_pnilp},
  we have that $G=PA$, where $P\in \Syl_p(G)$
  and $A$ is a normal $p$-complement and a Dedekind group.
  By Lemma~\ref{l:qclass_q8c2h}, we may assume that
  $A$ is abelian.
  By Theorem~\ref{t:qclass_nilp}, we can
  assume that $G$ is not nilpotent, and thus
  $\C_P(A) < P$.
  We assume that $G$ is not generalized dicyclic.
  By Lemmas~\ref{l:abab_1} and~\ref{l:c2ab},
  we can assume that
  $G = (PQ)\times B$
  and $A = Q \times B$,
  where $Q\in \Syl_q(G)$,
  and that $P$ is also abelian.
  Thus Lemma~\ref{l:blackburn2a} applies.
  Write $P= \erz{g} \times P_0$ and 
  $C= \C_P(A) = \erz{g^{p^c}} \times P_0$,
  as in that lemma.

  Let $z\in \R(G) \subseteq \erz{g^{p^c}}$ have order $p$
  and notice that we must have $z\in \NKer(G)$,
  since $1\neq \NKer(G)\leq \R(G)$.
  Choose $\mu\in \Lin(CB)$ such that $\mu(z)\neq 1$ 
  and $\mu$ has the maximal
  possible order $p^d n$, where $n$ equals the exponent of $B$
  and $p^d = \ord(g^{p^c})$.
  Choose $\lambda \in \Lin(Q)$, $\lambda\neq 1_Q$.
  Then $\chi = (\mu\times \lambda)^G \in \Irr(G)$ as in 
  Lemma~\ref{l:qsi_allab}.
  Since $z\notin \Ker(\chi)$,
  we must have 
  $m_{\rats}(\chi) = \chi(1)$.
  By the local-global principle (Lemma~\ref{l:sifacts}~\ref{it:globloc}),
  the Schur index of $\chi$ over some local field
  must equal $\chi(1)= \card{P:C} = p^c$.
  We claim that this local field must be $\rats_q$.
  If $p^c > 2$, then we must have
  $m_q(\chi)= \chi(1)$, 
  by Lemma~\ref{l:sifacts}~\ref{it:localsi_div}
  and Lemma~\ref{l:abbglocal}.
  If $p^c = 2$,
  then $\mu\in \Lin(CB)$ has order greater than $2$:
  Namely, when the exponent of $C \times B \leq  \Z(G)$ divides $2$,
  then $G$ is generalized dicyclic,
  and we assume that this is not the case.
  Thus $m_{\reals}(\chi) = 1$.
  Also, $\chi$ is induced from a subgroup of index $2$
  (which is not a $2$-group),  
  and thus $\chi$ remains irreducible after reducing mod~$2$.
  Thus $m_2(\chi) =1$ by Lemma~\ref{l:sifacts}~\ref{it:localsi_brauer}.
  Thus in every case, 
  we must have
  $m_q(\chi) = \chi(1)$.
  
  We can now apply Lemma~\ref{l:qsi_allab}.
  Notice that since $\mu(z)\neq 1$ by assumption,
  we have $\ord(\theta) = p^c \ord(\mu)$,
  as can be seen from the structure of $P$ 
  (Lemma~\ref{l:blackburn2a}), 
  and $\ord(\mu) = p^d \cdot n$,
  where $n$ equals the exponent of $B$.
  
  Let $k$ be the order of $q$ in $\ints/(p^d n)$
  and $\ell$ the order of $q$ in $\ints/(p^{c+d}n)$.
  By Lemma~\ref{l:qsi_allab}, we must have
  $\ell /k = \chi(1) = p^c$.
  This is only possible when 
  $(q^k-1)_p = p^d > 2$.
  By elementary number theory,
  the arithmetical conditions in Lemma~\ref{main:pq}
  must hold.
  Moreover, the $p$-part of the multiplicative order of $q$
  modulo $n$ must divide the multiplicative order of $q$ modulo $p^d$.
  As $n$ and $\card{B}$ have the same prime divisors,
  and $p$ does not divide $\card{B}$,
  this means that \ref{it:m_allab} in Theorem~\ref{main:qclass}
  holds.
\end{proof}

\section*{Acknowledgment}

I wish to thank Erik Friese for a thorough reading of this paper
and many useful remarks.
I gratefully acknowledge support by the DFG (Project: SCHU 1503/6-1).


\printbibliography

\end{document}